\documentclass[12pt]{amsart}
\usepackage{amsmath,amssymb,hyperref,geometry,graphicx,amsthm,tikz,enumitem,mathtools,textcomp}
\usepackage[all,cmtip]{xy}

\newtheorem{theorem}[]{Theorem}
\newtheorem{lemma}[equation]{Lemma}
\newtheorem{corollary}[equation]{Corollary}
\newtheorem{proposition}[equation]{Proposition}
\theoremstyle{definition}

\newtheorem{remark}[equation]{Remark}

\numberwithin{equation}{section}

\newcommand{\Z}{\mathbb{Z}}
\newcommand{\R}{\mathbb{R}}

\newcommand{\PL}{\mathrm{PL}}

\newcommand{\PE}{\mathrm{PE}}
\newcommand{\uPE}{\underline{\mathrm{PE}}}

\renewcommand{\L}{\mathrm{L}}

\renewcommand{\O}{\mathcal{O}}
\renewcommand{\phi}{\varphi}

\newcommand{\I}{\mathcal{I}}
\newcommand{\J}{\mathcal{J}}
\newcommand{\im}{\mathrm{im}}
\newcommand{\K}{\mathrm{K}}
\newcommand{\Span}{\mathrm{Span}}

\newcommand{\ZZ}{\Z}

\title{$\K$-rings of smooth toric varieties via piecewise-exponential functions}

\author[M.~Chan]{Melody Chan}
\address{Department of Mathematics, Brown University, United States of America}
\email{melody\_chan@brown.edu}

\author[E.~Clader]{Emily Clader}
\address{Department of Mathematics, San Francisco State University, United States of America}
\email{eclader@sfsu.edu}

\author[C.~Klivans]{Caroline Klivans}
\address{Department of Mathematics, Brown University, United States of America}
\email{caroline\_klivans@brown.edu}

\author[D.~Ross]{Dustin Ross}
\address{Department of Mathematics, San Francisco State University, United States of America}
\email{rossd@sfsu.edu}

\begin{document}

\begin{abstract}
    We describe an explicit presentation of the ring of integral piecewise-exponential functions on a unimodular fan as a quotient of the Stanley--Reisner ring of the fan. This gives rise to a presentation of $\K$-rings of smooth toric varieties that is parallel to the well-known presentation of integral Chow rings as quotients of Stanley--Reisner rings.
\end{abstract}

\maketitle

\vspace{-.5cm}
\section{Introduction}

The purpose of this note is to provide a  presentation of the Grothendieck $\K$-ring $K(X_\Sigma)$ of a smooth toric variety as a quotient of the Stanley--Reisner ring of the corresponding fan $\Sigma$.  There are closely-related results and certain special cases in the literature, as we discuss below, but to our knowledge, our main results (Theorems~\ref{thm2} and~\ref{thm1}) are new.

Let $\Sigma$ be a unimodular fan in a real vector space $N_\R$ with lattice $N\subseteq N_\R$, and consider the smooth toric variety $X_\Sigma$ associated to $\Sigma$. For each ray $\rho\in\Sigma(1)$, let $u_\rho\in N$ denote its primitive generator and let $\O(D_\rho)$ denote the line bundle on $X_\Sigma$ whose first Chern class is the toric divisor $D_\rho$.  Our main result is the following.

\begin{theorem}\label{thm2}
    For any smooth toric variety $X_{\Sigma}$, the ring homomorphism
    \begin{align*}
        \Z[x_\rho\mid \rho\in\Sigma(1)] &\longrightarrow K(X_\Sigma)\\ x_\rho&\mapsto 1-[\O(D_\rho)]
    \end{align*}
    induces an isomorphism
    \[
    \frac{\Z[x_\rho\mid\rho\in\Sigma(1)]}{\I+\J}\cong K(X_\Sigma),
    \]
where
    \[
    \I=\langle x_{\rho_1}\cdots x_{\rho_k}\mid\rho_1,\dots,\rho_k\text{\normalfont\ do not lie on a common cone of }\Sigma\rangle
    \]
    and
    \[
    \J=\Big\langle\prod_{\substack{\rho \text{ such that}\\ \langle m,u_\rho\rangle>0}}(1-x_\rho)^{\langle m,u_\rho\rangle}-\prod_{\substack{\rho \text{ such that}\\ \langle m,u_\rho\rangle<0}}(1-x_\rho)^{-\langle m,u_\rho\rangle}\;\Big|\; m\in N^\vee \Big\rangle.
    \]
\end{theorem}

The Stanley--Reisner ring of $\Sigma$ is the quotient $\Z[x_\rho\mid \rho\in \Sigma(1)]/\mathcal{I}$, so Theorem~\ref{thm2} indeed presents $K(X_{\Sigma})$ as a quotient of this ring.

The proof of Theorem~\ref{thm2} 
leverages the known isomorphism \cite{BrionVergne,Merkurjev}\begin{equation}\label{eq:K-uPE}K(X_\Sigma)\cong\uPE(\Sigma),\end{equation} where $\uPE(\Sigma)$ denotes the ring of integral piecewise-exponential functions on $\Sigma$, modulo the relations generated by differences of exponentials of linear functions; see Section~\ref{sec:prelim} for the precise definition.  Specifically, we deduce Theorem~\ref{thm2} as a consequence of the identification~\eqref{eq:K-uPE} and the following main combinatorial result of the paper.

\begin{theorem}\label{thm1}
    The ring homomorphism
    \begin{align*}
    \phi\colon\Z[x_\rho\mid\rho\in\Sigma(1)]&\rightarrow\uPE(\Sigma)\\
    \phi(x_\rho) &= 1-[e^{\delta_\rho}]
    \end{align*}induces an isomorphism
    \[
    \frac{\Z[x_\rho\mid\rho\in\Sigma(1)]}{\I+\J}\cong\uPE(\Sigma).
    \]
\end{theorem}

The presentation of Theorem~\ref{thm2} was previously proved for special types of smooth toric varieties: in the case that $X_\Sigma$ is projective, it was proved by Sankaran and Uma \cite[Theorem~1.2]{SankaranUma}, and this was later extended to complete fans by Sankaran \cite[Theorem~2.2]{Sankaran}; in the case that $\Sigma$ is the Bergman fan of a matroid, an equivalent formulation of Theorem~\ref{thm2} was proved by Larson, Li, Payne, and Proudfoot \cite[Theorem 5.2]{LLPP}, and their proof readily extends to any fan for which $K(X_\Sigma)$ is torsion free, thereby generalizing the complete case. Since $K(X_\Sigma)$ need not be torsion free in general (see \cite[Example~2.12]{CCKR}), a new argument is needed to prove Theorem~\ref{thm2} in full generality.

On the other hand, an analogue of Theorem~\ref{thm2} after inverting $1-x_\rho$ for every $\rho\in\Sigma(1)$ was proved by Vezzosi and Vistoli over two decades ago.  In our notation, \cite[Theorem~6.4]{VezzosiVistoli} gives a presentation of the {\em equivariant} K-ring of $X_\Sigma$, with respect to its dense torus $T$, as 
\begin{equation}\label{eq:vv}
K_T(X_\Sigma) \cong \frac{\ZZ[x_\rho \mid \rho \in \Sigma(1)]_t}{\mathcal{I}_t},
\end{equation}
where $t = \prod_{\rho\in\Sigma(1)}(1-x_\rho)$ and the subscript denotes localization.  A result of Merkujev \cite{Merkurjev} then allows for a presentation of the non-equivariant K-ring $K(X_\Sigma)$ as an appropriate quotient of~\eqref{eq:vv}, so one obtains a presentation of $K(X_\Sigma)$ as a quotient of the Laurent polynomial ring in variables $1-x_\rho$.  What is novel about our work is that we achieve a presentation as a quotient of the {\em polynomial} ring in these variables---the polynomial presentation is not at all obvious from the Laurent presentation. The main argument is contained in Proposition~\ref{prop:unlocalize} below.

Despite the connections to previous algebro-geometric work, we structured this note to include a self-contained and rather elementary proof of Theorem~\ref{thm1}, given entirely within the context of the combinatorics and algebra of piecewise-exponential functions.  Providing proofs within the settings of piecewise-exponential or piecewise-polynomial functions on fans, without invoking theorems in algebraic geometry, is in the spirit of previous combinatorial work including Billera's study of piecewise-polynomial functions on simplicial complexes \cite{Billera} and McMullen's polytope algebras \cite{McMullen}.  Indeed, some of our initial motivation for the paper \cite{CCKR} that led to this work came from an attempt to understand aspects of McMullen's polytope algebra in the setting of incomplete polyhedral fans.

Lastly, we note that, in addition to the works cited elsewhere in this paper, there exists a broad and deep collection of literature on the K-rings of toric varieties and their generalizations; see \cite{MorelliK,SankaranUma2,BorisovHorja,Uma,JoshuaKrishna,DKU,DKU2,Uma2,SarkarUma,Uma3}.

\subsection*{Acknowledgments}

An incomplete proof of Theorem~\ref{thm1} appeared in an early draft of our paper \cite{CCKR}, and we thank Matt Larson for pointing out the error; the subtleties necessary to patch the proof are what led to this paper.  We also thank an anonymous referee for suggesting a simplification of our original proof of Proposition~\ref{prop:unlocalize}.  This work began while E.C. and D.R. were visiting Brown University for sabbatical; they warmly acknowledge the Brown Department of Mathematics for their support and for providing a welcoming and stimulating research environment. M.C.~was supported by NSF CAREER DMS--1844768, FRG DMS--2053221, and DMS--2401282. E.C. was supported by NSF CAREER DMS--2137060.  D.R. was supported by NSF DMS--2302024 and DMS--2001439.

\section{Preliminary conventions and definitions}\label{sec:prelim}

Throughout, $N$ denotes a finitely-generated free abelian group, and $N_\R=\R\otimes_\Z N$ the associated vector space. We let $M$ and $M_\R$ denote the duals of $N$ and $N_\R$, with pairing written $\langle-,-\rangle$. Throughout the paper, $\Sigma$ denotes a unimodular rational polyhedral fan in $N_\R$ of dimension $d$, the rays of which are denoted $\Sigma(1)$, and for each cone $\sigma\in\Sigma$, the rays of $\sigma$ are denoted $\sigma(1)$. Given a ray $\rho\in\Sigma(1)$, we let $u_\rho\in N$ denote the primitive ray generator of $\rho$. The support of a fan $\Sigma$ is denoted $|\Sigma|$.  For cones $\tau,\sigma \in \Sigma$, write $\tau \preceq \sigma$ if $\tau$ is a face of $\sigma$. 

A function $\ell:|\Sigma|\rightarrow\R$ is said to be \textbf{integral linear on $\Sigma$} if $\ell=m|_{|\Sigma|}$ for some $m\in M$. The group of integral linear functions on $\Sigma$ is denoted $\L(\Sigma)$. A function $f:|\Sigma|\rightarrow\R$ is said to be \textbf{integral piecewise-linear on $\Sigma$} if, for every $\sigma\in\Sigma$, there exists $m\in M$ such that $f|_\sigma=m|_\sigma$. The group of integral piecewise-linear functions on $\Sigma$ is denoted $\PL(\Sigma)$. A function $F:|\Sigma|\rightarrow\R$ is said to be \textbf{integral piecewise-exponential on $\Sigma$} if, for every $\sigma\in\Sigma$, there exist $a_1,\dots,a_n\in\Z$ and $m_1,\dots,m_n\in M$ such that
\begin{equation}
\label{eq:Fsigma}
F|_\sigma=\big(\sum_{i=1}^na_ie^{m_i}\big)\big|_\sigma.
\end{equation}
The ring of integral piecewise-exponential functions on $\Sigma$ is denoted $\PE(\Sigma)$. Lastly, define the quotient ring
\[
\uPE(\Sigma)=\frac{\PE(\Sigma)}{\langle e^\ell-1\mid\ell\in\L(\Sigma) \rangle}.
\]
We typically denote elements of $\uPE(\Sigma)$ as equivalence classes $[F]$ with $F\in\PE(\Sigma)$.

\begin{remark}
The above definition differs slightly from the definition of $\PE(\Sigma)$ in our earlier work \cite{CCKR}: in that case, we asked that an integral piecewise-exponential function $F$ be globally expressible as
\[F = \sum_{i=1}^n a_i e^{f_i}\]
where $f_i \in \PL(\Sigma)$.  It is not immediately obvious that the cone-by-cone linear functions $m_i$ in \eqref{eq:Fsigma} patch to give piecewise-linear functions, so our earlier definition is a priori stronger than the one we give here.  However, when $\Sigma$ is unimodular, the two definitions coincide by Lemma~\ref{lem:PEgens} below.  We choose to give the weaker definition in this paper for consistency with other related work such as \cite{AndersonPayne, Brion, BrionVergne}, which in particular suggests that it is the more appropriate definition in the non-unimodular context.
\end{remark}

For each $\rho\in\Sigma(1)$, the {\bf Courant function} $\delta_\rho\in\PL(\Sigma)$ is the unique piecewise-linear function taking value $1$ at $u_\rho$ and vanishing along all other rays of $\Sigma$. Observe that the group $\PL(\Sigma)$ is freely generated as a group by the Courant functions. The next result implies that exponentials of piecewise-linear functions span the ring $\PE(\Sigma)$.

\begin{lemma}
\label{lem:PEgens}
Let $\Sigma$ be a unimodular fan.
The ring $\PE(\Sigma)$ is generated by $\{e^{\pm\delta_\rho}\mid\rho\in\Sigma(1)\}$.
\end{lemma}

\begin{proof}
First, suppose that $\sigma\in\Sigma$ is a cone and $\ell_1,\dots,\ell_j,\ell_1',\dots,\ell_k'\in\L(\sigma)$ are linear functions on $\sigma$ such that we have the following equality of functions on $\sigma$:
\[
e^{\ell_1}+\dots+e^{\ell_j}=e^{\ell_1'}+\dots+e^{\ell_k'};
\]
we claim that $j=k$ and there is an equality of multisets $\{\ell_1,\dots,\ell_j\}=\{\ell_1',\dots,\ell_k'\}$. For if not, then we could choose nonzero $u\in\sigma$ where $\{\ell_1(u),\dots,\ell_j(u)\}\neq\{\ell_1'(u),\dots,\ell_k'(u)\}$, and upon reordering, we could assume that each multiset is listed in nonincreasing order and that $\{\ell_1(u),\dots,\ell_j(u)\}>\{\ell_1'(u),\dots,\ell_k'(u)\}$ in lexicographic ordering. But then
\[
e^{\ell_1(\lambda u)}+\dots+e^{\ell_j(\lambda u)}>e^{\ell_1'(\lambda u)}+\dots+e^{\ell_k'(\lambda u)}\;\;\;\text{ for all }\;\;\;\lambda\gg 0. 
\]

From the previous paragraph, it follows that for each $F\in\PE(\Sigma)$ and each $\sigma\in\Sigma$, there exists a unique ordered pair of multisets of linear functions $\{\ell_{\sigma,1},\dots,\ell_{\sigma,j}\}$ and $\{\ell_{\sigma,1}',\dots,\ell_{\sigma,k}'\}$ in $\L(\sigma)$ 
that are disjoint from each other
and such that
\[
F|_\sigma=\sum_{i=1}^je^{\ell_{\sigma,i}}-\sum_{i=1}^k e^{\ell_{\sigma,i}'}.
\]
Define
\[
F_\sigma=\sum_{i=1}^je^{\sum_{\rho\in\sigma(1)}\ell_{\sigma,i}(u_\rho)\delta_\rho}-\sum_{i=1}^k e^{\sum_{\rho\in\sigma(1)}\ell_{\sigma,i}'(u_\rho)\delta_\rho}.
\]
In other words, $F_\sigma$ agrees with $F$ on $\sigma$ and interpolates (in the unique piecewise-exponential way) to the function that takes constant value $j-k$ on all rays of $\Sigma$ not contained in $\sigma$. Notice that $F_\sigma$ is in the subring generated by $\{e^{\pm\delta_\rho}\mid\rho\in\Sigma(1)\}$. Then an inclusion-exclusion argument shows that 
\[F = \sum_{\sigma\in \Sigma} \mu(\sigma) F_\sigma,\]
where $\mu \colon \Sigma \to \ZZ$ is the unique assignment of integers to the cones of $\Sigma$ that satisfies $\sum_{\sigma \succeq \tau} \mu(\sigma) = 1$ for each $\tau \in \Sigma$.  In particular, $F$ is in the $\ZZ$-linear span of 
$\{F_\sigma \; | \; \sigma\in\Sigma\}$.
\end{proof}

\section{Proofs of Theorems~\ref{thm2} and~\ref{thm1}}

We begin with the following surjectivity result.

\begin{proposition}
\label{cor:invert}
The ring homomorphism
\begin{align*}
\varphi\colon\Z[x_\rho\mid\rho\in\Sigma(1)]&\longrightarrow\uPE(\Sigma)\\
x_\rho &\mapsto 1-[e^{\delta_\rho}]
\end{align*}
is surjective.
\end{proposition}

\begin{remark}
If one prefers an algebraic geometry proof of Proposition~\ref{cor:invert}, it can be derived using the coniveau filtration on $K(X_\Sigma)$; see \cite[Lemma~2.7]{LLPP}.
\end{remark}

\begin{remark}
From an algebro-geometric perspective, it is slightly more natural to consider the homomorphism $x_\rho \mapsto 1- [e^{-\delta_\rho}]$, since the latter corresponds under the isomorphism $\uPE(\Sigma) \cong K(X_\Sigma)$ to the structure sheaf $\O_{D_\rho}$ of the toric divisor $D_\rho$.  This sign convention is inconvenient for the combinatorial arguments to follow, but we note that it differs from our sign choice only by the involution $\mathcal{E} \mapsto \mathcal{E}^{\vee}$ on $K(X_{\Sigma})$.
\end{remark}

\begin{proof}[Proof of Proposition~\ref{cor:invert}]
Since $\uPE(\Sigma)$ is generated as a ring by $\{[e^{\pm\delta_\rho}]\mid\rho\in\Sigma(1)\}$, it suffices to show that $[e^{\pm\delta_\rho}]\in\im(\phi)$ for each $\rho\in\Sigma(1)$, and since $\phi(1-x_\rho)=[e^{\delta_\rho}]$, we really only need to argue that $[e^{-\delta_\rho}]\in\im(\phi)$. To accomplish this, we prove below that $(1-[e^{\delta_\rho}])^{d+1}=0$, from which it follows that
\[
[e^{-\delta_\rho}]=\frac{1}{1-(1-[e^{\delta_\rho}])}=\sum_{i=1}^d(1-[e^{\delta_\rho}])^i=\phi\Big(\sum_{i=1}^dx_\rho^i\Big).
\]
In fact, we prove a stronger nilpotency result in Proposition~\ref{prop:nilpotency} below, whose proof then finishes the proof of Proposition~\ref{cor:invert}.
\end{proof}

\begin{proposition}\label{prop:nilpotency}
If $\Sigma$ is a unimodular fan of dimension $d$ and $f \in \PL(\Sigma)$, then
\[
([e^f] - 1)^{d+1} = 0 \in \uPE(\Sigma).
\]
\end{proposition}

To prove Proposition~\ref{prop:nilpotency}, we use a pullback operation on rings of piecewise-exponential functions. Let $\Sigma$ and $\Sigma'$ be unimodular fans in ambient spaces $N_\R$ and $N'_\R$, and let $\psi: \Sigma \rightarrow \Sigma'$ be a morphism of fans, which is a linear map $\psi\colon N_\R \rightarrow N'_\R$ taking $N$ to $N'$ and taking each cone of $\Sigma$ into a cone of $\Sigma'$. The {\bf pullback homomorphism} induced by $\psi$ is defined by
\begin{align*}
\psi^*: \uPE(\Sigma') &\rightarrow \uPE(\Sigma)\\
[F] &\mapsto [F \circ \psi].
\end{align*}
On generators $[e^f]$ of $\uPE(\Sigma')$, the pullback homomorphism is given by $\psi^*([e^f]) = [e^{f \circ \psi}]$.  Note that $f \circ \psi \in \PL(\Sigma)$ because $\psi$ is a morphism of fans, and from here, the fact that $\psi^*$ is well-defined follows from linearity of $\psi$, which implies that $\ell \circ \psi \in \L(\Sigma)$ whenever $\ell \in \L(\Sigma')$.

We will be particularly interested in two special cases of the pullback homomorphism. The first is when we have an inclusion of a subfan.

\begin{lemma}
\label{lem:inclusion*}
Let $\Sigma$ be a unimodular fan, and let $i: \Sigma' \hookrightarrow \Sigma$ be the inclusion of a subfan. Then $\ker(i^*)=\big\langle [F]\mid F\in\PE(\Sigma)\text{ and }F|_{|\Sigma'|}=0 \big\rangle$.

\end{lemma}

\begin{proof}
By unimodularity of $\Sigma$, any linear function $\ell'$ on $\Sigma'$ can be extended to a linear function $\ell$ on $\Sigma$.  Therefore, in the commuting diagram with short exact rows
\[
\xymatrix@R=6mm@C=6mm{0\ar[r] & \langle e^\ell - 1 \mid \ell\in \L(\Sigma)\rangle \ar[r]\ar[d] & \PE(\Sigma) \ar[r]\ar[d] & \uPE(\Sigma)\ar[r] \ar[d]& 0 \\ 0\ar[r] & \langle e^{\ell'} - 1 \mid \ell' \in \L(\Sigma')\rangle \ar[r] & \PE(\Sigma') \ar[r] & \uPE(\Sigma')\ar[r] & 0, }
\]
the first vertical arrow is surjective.  By the snake lemma, the map
\[
\ker(\PE(\Sigma)\to \PE(\Sigma')) \rightarrow \ker(\uPE(\Sigma)\to \uPE(\Sigma'))
\]
is surjective, proving the lemma.
\end{proof}

The second setting in which we will be interested in studying the pullback homomorphism relates to star fans. Before presenting the lemma, we first set up notation. Given a unimodular fan $\Sigma$ and a cone $\sigma\in\Sigma$, we define the \textbf{neighborhood} of $\sigma$ to be the subfan 
\[
\Sigma^{(\sigma)}=\{\tau\in\Sigma\mid\tau\preceq\sigma'\text{ for some }\sigma'\succeq\sigma\}.
\]
In words, $\Sigma^{(\sigma)}$ collects all cones that contain $\sigma$, along with the faces of those cones. The \textbf{star} of $\sigma$ is the fan in the quotient vector space $N_\R^\sigma=N_\R/\Span(\sigma)$ defined by
\[
\Sigma^\sigma=\{\overline\tau\mid\tau\in\Sigma^{(\sigma)}\},
\]
where $\overline\tau$ denotes the quotient of $\tau$. Note that $\Sigma^\sigma$ is unimodular with respect to the lattice $N^\sigma=N/(N\cap\Span(\sigma))$. The following is the key result regarding star fans.

\begin{lemma}\label{lem:QSigma*}
Let $\Sigma$ be a unimodular fan, $\sigma \in \Sigma$ a cone, and $q\colon \Sigma^{(\sigma)} \rightarrow \Sigma^\sigma$ the quotient map. Then $q^*\colon \uPE(\Sigma^\sigma)\rightarrow \uPE(\Sigma^{(\sigma)})$ is an isomorphism.
\end{lemma}

\begin{proof}
For any $f \in \PL(\Sigma^{(\sigma)})$, choose an integral linear function $\ell_{f, \sigma} \in \L(\Sigma^{(\sigma)})$ that agrees with $f$ on $\sigma$.  Then 
\[
[e^f] = [e^{f-\ell_{f,\sigma}}] \in \uPE(\Sigma^{(\sigma)}).
\]
Since $f - \ell_{f,\sigma}$ vanishes on $\Span(\sigma)$, it gives a well-defined function 
$\overline{f - \ell_{f,\sigma}} \in \PL(\Sigma^\sigma)$ for which
\[
f - \ell_{f,\sigma} = \overline{f - \ell_{f,\sigma}} \circ q.
\]
If $\ell_{f,\sigma},\ell_{f,\sigma}'\in L(\Sigma^{(\sigma)})$ both agree with $f$ on $\sigma$, then $\overline{f - \ell_{f,\sigma}}-\overline{f - \ell'_{f,\sigma}}=\overline{\ell'_{f,\sigma}-\ell_{f,\sigma}}$ is linear on $\Sigma^\sigma$, and it follows that we obtain a well-defined inverse to $q^*$:
\begin{align*}
\uPE(\Sigma^{(\sigma)}) &\rightarrow \uPE(\Sigma^\sigma)\\
[e^f] &\mapsto \big[e^{\overline{f - \ell_{f,\sigma}}}\big]\qedhere
\end{align*}
\end{proof}

With these lemmas in place, we are now prepared to prove Proposition~\ref{prop:nilpotency}.

\begin{proof}[Proof of Proposition~\ref{prop:nilpotency}]
We prove that $\big([e^f] - 1\big)^{d+1}=0$ for all $f\in\PL(\Sigma)$ by induction on the dimension $d$ of $\Sigma$. The base case is $d=0$, in which case every piecewise-linear function $f$ on $\Sigma$ is the zero function, so $e^f = e^0 = 1$ and the result is immediate.

Suppose, by induction, that the result holds for all unimodular fans of dimension $d' < d$, and let $\Sigma$ be a unimodular fan of dimension $d$.  For any $f \in \PL(\Sigma)$ and $\sigma\in\Sigma$, define $f_\sigma\in\PL(\Sigma)$ on ray generators by
\[
f_\sigma(u_\rho) = \begin{cases} f(u_\rho) & \text{ if } \rho \in \sigma(1)\\ 0 & \text{ otherwise.} \end{cases}
\]
Inclusion-exclusion produces a function $\mu:\Sigma\rightarrow\Z$ such that, for every $f\in\PL(\Sigma)$, we have
\[
e^f = \sum_{\sigma \in \Sigma} \mu(\sigma) e^{f_\sigma}.
\]
\noindent Applying this equality to multiples $if$ of $f$, we compute
\begin{align*}
([e^f] - 1)^{d+1} &= \sum_{i=0}^{d+1} \binom{d+1}{i} (-1)^{d+1-i} [e^{if}]\\
&=\sum_{\sigma \in \Sigma} \mu(\sigma) \bigg(\sum_{i=0}^{d+1} \binom{d+1}{i} (-1)^{d+1-i} [e^{if_\sigma}]\bigg)\\
&=\sum_{\sigma \in \Sigma} \mu(\sigma) ([e^{f_\sigma}] - 1)^{d+1}.
\end{align*}
Thus, in order to prove the proposition, it is sufficient to prove that
\begin{equation}
    \label{eq:nilgoalsigma}
    ([e^{f_\sigma}] - 1)^{d+1} = 0 \in \uPE(\Sigma)
\end{equation}
for every $f\in\PL(\Sigma)$ and $\sigma \in \Sigma$.  This holds when $\sigma$ is the zero-dimensional cone---since then $f_\sigma = 0$ and hence $e^{f_\sigma} = e^0 = 1$---so we assume in what follows that $\sigma$ has positive dimension.

Given that $\uPE(\Sigma^{(\sigma)}) \cong \uPE(\Sigma^\sigma)$ by Lemma~\ref{lem:QSigma*}, the fact that $\text{dim}(\Sigma^\sigma) < d$ and the inductive hypothesis imply that
\[
i^*([e^{f_\sigma}] - 1)^d =0 \in \uPE(\Sigma^{(\sigma)}).
\]
From here, Lemma~\ref{lem:inclusion*} implies that we can write
\[
([e^{f_\sigma}] - 1)^d = [F]
\]
for a function $F\in\PL(\Sigma)$ that vanishes on $|\Sigma^{(\sigma)}|$.  Therefore,
\[
([e^{f_\sigma}] - 1)^{d+1} = [F]\cdot ([e^{f_\sigma}] - 1)=[F\cdot(e^{f_\sigma}-1)]\in\uPE(\Sigma).
\]
Notice, here, that $e^{f_{\sigma}}-1$ vanishes on $|\Sigma| \setminus |\Sigma^{(\sigma)}|$, since $f_\sigma$ vanishes on this domain.  Thus, we have a product of a function $F$ that vanishes on $|\Sigma^{(\sigma)}|$ and a function $e^{f_{\sigma}}-1$ that vanishes on the complement of $|\Sigma^{(\sigma)}|$, so their product vanishes on all of $\Sigma$.  That is,
\[
([e^{f_\sigma}] - 1)^{d+1}=0\in\uPE(\Sigma),
\]
completing the induction step and thus the proof of the proposition.
\end{proof}

Having now proven that $\varphi$ is surjective, it remains to  argue that 
\[
\ker(\phi)=\I+\J,
\] where $\I$, and $\J$ are as in Theorem~\ref{thm1}. To begin, we lift the ring homomorphism $\varphi$ to a ring homomorphism $\widetilde\phi:\Z[x_\rho\mid\rho\in\Sigma(1)]\rightarrow\PE(\Sigma)$.

\begin{proposition}\label{prop:kerneltilde}
    Let $\widetilde\phi:\Z[x_\rho\mid\rho\in\Sigma(1)]\rightarrow\PE(\Sigma)$ be the ring homomorphism determined by $\widetilde\phi(x_\rho)=1-e^{\delta_\rho}$. Then $\ker(\widetilde\phi)=\I$.
\end{proposition}

\begin{proof}
First, we prove that $\I\subseteq\ker(\widetilde\phi)$. Suppose that $\rho_1, \ldots, \rho_k \in \Sigma(1)$ do not span a cone of $\Sigma$. Then
\[ 
\widetilde\phi\left(x_{\rho_1} \cdots x_{\rho_k}\right) = \prod_{i=1}^k(1 - e^{\delta_{\rho_i}} )=\sum_{I \subseteq \{\rho_1, \ldots, \rho_k\}} (-1)^{|I|} e^{\sum_{\rho \in I} \delta_\rho}
.\]
This is the function $|\Sigma| \rightarrow \R$ that sends each $w \in |\Sigma|$ to 
\begin{equation}
    \label{eq:wantiszero}
    \sum_{I \subseteq \{\rho_1, \ldots, \rho_k\}} (-1)^{|I|} e^{\sum_{\rho\in I \cap \sigma_w(1)}\delta_\rho(w)},
\end{equation}
where $\sigma_w\in\Sigma$ is the unique cone containing $w$ in its relative interior. The fact that $\rho_1, \ldots, \rho_k$ do not span a cone implies that there exists some element in $\{\rho_1, \ldots, \rho_k\} \setminus \sigma_w(1)$, which we assume without loss of generality is $\rho_k$.  For any subset $I \subseteq \{\rho_1, \ldots, \rho_{k-1}\}$, we have
\[
I \cap \sigma_w(1) = (I \cup \{\rho_k\}) \cap \sigma_w(1),
\]
and the summands in \eqref{eq:wantiszero} indexed by $I$ and by $I \cup \{\rho_k\}$ appear in \eqref{eq:wantiszero} with opposite signs.  It follows that \eqref{eq:wantiszero} equals zero, showing that $\I\subseteq\ker(\widetilde\phi)$.

In light of the inclusion $\I\subseteq\ker(\widetilde\phi)$, there is a well-defined homomorphism
\[
\overline{\phi}: \frac{\Z[x_\rho \; | \; \rho \in \Sigma(1)]}{\I} \rightarrow \PE(\Sigma).
\]
To complete the proof of the proposition, we argue that $\overline\phi$ is injective. Note that the domain of $\overline{\phi}$ is the Stanley--Reisner ring of $\Sigma$, a $\Z$-basis of which is given by the monomials of the form
\[
\prod_{\rho\in\sigma(1)}[x_\rho]^{a_\rho}
\]
as $\sigma$ ranges over all cones of $\Sigma$ and $a_\rho$ over all positive integers.  Thus, to prove injectivity of $\overline{\phi}_\Sigma$, it suffices to prove that no nontrivial $\Z$-linear combination of these cone monomials is mapped to zero, which is equivalent to the claim that the elements
\[
\prod_{\rho \in \sigma(1)} (1 - e^{\delta_\rho})^{a_\rho}
\]
are $\Z$-linearly independent in $\PE(\Sigma)$.

Toward a contradiction, suppose that we have a nontrivial linear dependence
\begin{equation}\label{eq:dependence}
0=\sum_{a\in S} \lambda_{a}\prod_{\rho \in \Sigma(1)} (1 - e^{\delta_\rho})^{a_\rho}
\end{equation}
where $S$ is a nonempty finite set of tuples $a\in\Z_{\geq 0}^{\Sigma(1)}$, each of which has nonzero entries supported on indices corresponding to the rays of some cone $\sigma_a\in\Sigma$, and where each $\lambda_a$ is a nonzero integer. Choose an ordering of the rays $\Sigma(1)=\{\rho_1,\dots,\rho_m\}$, order the tuples $a\in S$ lexicographically with respect to this ray ordering, and let $\widehat a$ be the maximal element of $S$. Suppose that $\sigma_{\widehat a}\in\Sigma(k)$, and, without loss of generality, assume that the first $k$ rays $\rho_1,\dots,\rho_k\in\Sigma(1)$ are the rays of $\sigma_{\widehat a}$. (To do so, move rays not in $\sigma_{\widehat a}(1)$ to the end of the ordering, without changing the ordering of the rays in $\sigma_{\widehat a}(1)$, which does not affect the maximality of $\widehat a$.) We now show that \eqref{eq:dependence} implies that $\lambda_{\widehat a}=0$, a contradiction.

Consider positive values $\epsilon_{\rho_1},\dots,\epsilon_{\rho_k}>0$, and evaluate the right-hand side of \eqref{eq:dependence} at 
\[
u=\epsilon_{\rho_1}u_{\rho_1}+\epsilon_{\rho_2}u_{\rho_2}+\dots+\epsilon_{\rho_k}u_{\rho_k}\in\sigma_{\widehat a}.
\]
The term in the expansion of the $a$th summand of \eqref{eq:dependence} with the largest power of $e$ is
\[
\lambda_a(-1)^{\sum a_\rho}e^{\sum\epsilon_\rho a_\rho},
\]
where both sums appearing in exponents are over $\rho\in\sigma_{a}(1)\cap\sigma_{\widehat a}(1)$. As long as we choose $\epsilon_{\rho_j},\dots,\epsilon_{\rho_k}$ to be sufficiently small with respect to $\epsilon_{\rho_1},\dots,\epsilon_{\rho_{j-1}}$ for every $j=2,\dots,k$, the maximality of $\widehat a$ implies that the unique maximal power of $e$ appearing in all of these summands is the one associated to $\widehat a$:
\begin{equation}\label{eq:maximalterm}
\lambda_{\widehat{a}}(-1)^{\sum \widehat a_\rho}e^{\sum\epsilon_\rho\widehat a_\rho}.
\end{equation}
Thus, evaluating the right-hand side of \eqref{eq:dependence} at arbitrarily large multiples of $u$, the maximal term \eqref{eq:maximalterm} outgrows all others, from which \eqref{eq:dependence} implies that $\lambda_{\widehat a}=0$, as desired.
\end{proof}

The homomorphism $\widetilde\phi$ considered above is not surjective, simply because $e^{-\delta_\rho}$ is not in the image. However, if we localize the domain at the element 
\[
t=\prod_{\rho\in\Sigma(1)}(1-x_\rho),
\]
noting that $\widetilde\phi(t)$ is invertible in $\PE(\Sigma)$, we obtain the following extension of Proposition~\ref{prop:kerneltilde}.

\begin{corollary}\label{cor:surj1}
    The ring homomorphism $\widetilde\phi_t:\Z[x_\rho\mid\rho\in\Sigma(1)]_t\rightarrow\PE(\Sigma)$ determined by $\widetilde\phi_t(x_\rho)=1-e^{\delta_\rho}$ is surjective with $\ker(\widetilde\phi_t)=\I_t$.
\end{corollary}

\begin{proof}
    Notice that $1-x_\rho$ is invertible in $\Z[x_\rho \; | \; \rho\in\Sigma(1)]_t$ and $\widetilde\phi_t((1-x_\rho)^{\pm 1})=e^{\pm\delta_\rho}$. Thus, all of the generators of $\PE(\Sigma)$ are in the image of $\widetilde\phi_t$, so $\widetilde\phi_t$ is surjective. The computation of the kernel then follows from localizing the exact sequence obtained from Proposition~\ref{prop:kerneltilde}:
    \[
    0\rightarrow\I\rightarrow\Z[x_\rho\mid\rho\in\Sigma(1)]\stackrel{\widetilde\phi}{\rightarrow}\PE(\Sigma).\qedhere
    \]
\end{proof}

From the previous result, we then obtain the following approximation of Theorem~\ref{thm1}.

\begin{corollary}\label{cor:surj2}
    The ring homomorphism $\phi_t:\Z[x_\rho \; | \; \rho\in\Sigma(1)]_t\rightarrow\uPE(\Sigma)$ determined by $\phi_t(x_\rho)=1-e^{\delta_\rho}$ is surjective with $\ker(\phi_t)=\I_t+\J_t$.
\end{corollary}

\begin{proof}
    Corollary~\ref{cor:surj1} yields an isomorphism
    \[
    \frac{\Z[x_\rho \; | \; \rho\in\Sigma(1)]_t}{\I_t}\cong\PE(\Sigma),
    \]
    and this isomorphism identifies $\J_t$ with $\langle e^\ell-1\mid\ell\in\L(\Sigma)\rangle=\ker(\PE(\Sigma)\rightarrow\uPE(\Sigma))$. 
\end{proof}

We note that Corollary~\ref{cor:surj2} is equivalent to the presentation of equivariant K-rings due to Vezzosi and Vistoli \cite[Theorem~6.4]{VezzosiVistoli}. The next result is the bridge between their presentation and Theorem~\ref{thm2}.

\begin{proposition}\label{prop:unlocalize}
The ring homomorphism $\phi:\Z[x_\rho\mid\rho\in\Sigma(1)]\rightarrow\uPE(\Sigma)$ determined by $\phi(x_\rho)=1-[e^{\delta_\rho}]$ has $\ker(\phi)=\I+\J$.
\end{proposition}

\begin{proof}
First, note that $\ker(\varphi) \supseteq \I$ by Proposition~\ref{prop:kerneltilde}, and it is straightforward to check that $\ker(\varphi) \supseteq \J$, as well.  Therefore, $\varphi$ factors through the quotient
\begin{equation}
    \label{eq:factorphi}
    \Z[x_\rho \; | \; \rho \in \Sigma(1)] \rightarrow \frac{\Z[x_\rho \; | \; \rho \in \Sigma(1)]}{\I + \J}.
\end{equation}

We claim that the natural map
\begin{equation}
    \label{eq:localization}
    \frac{\Z[x_\rho \; | \; \rho \in \Sigma(1)]}{\I + \J} \rightarrow \left(\frac{\Z[x_\rho \; | \; \rho \in \Sigma(1)]}{\I + \J} \right)_t
\end{equation}
is an isomorphism.  In order to prove this, it suffices to check that, for any ray $\rho_0 \in \Sigma(1)$, the element
\[[1-x_{\rho_0}] \in \frac{\Z[x_\rho \; | \; \rho \in \Sigma(1)]}{\I + \J} \]
is a unit.

Toward this end, let $\rho_0 \in \Sigma(1)$ be fixed. It is a standard result of Stanley--Reisner theory that the ideal $\I$ can be expressed as
\[\I = \bigcap_{\sigma \text{ maximal}}\I_\sigma,\]
where the intersection is over all maximal cones of $\Sigma$ and
\[\I_\sigma \coloneq \langle x_\rho \; | \; \rho \in \Sigma(1) \setminus \sigma(1) \rangle.\]
We now claim that $x_{\rho_0} \in \I_\sigma + J$ for all maximal cones $\sigma$.  If $\rho_0$ is not contained in $\sigma$, we clearly have $x_{\rho_0} \in \I_\sigma$ and therefore $x_{\rho_0} \in \I_\sigma + \J$.  On the other hand, for maximal cones $\sigma$ that do contain $\rho_0$, the fact that $\Sigma$ is unimodular allows one to choose $m \in M$ such that $\langle m, u_{\rho_0}\rangle =1$ and $\langle m, u_\rho \rangle = 0$ for all $\rho \in \sigma(1) \setminus \rho_0$.  Then, in the quotient
\[\frac{\Z[x_\rho \; | \; \rho \in \Sigma(1)]}{\I_\sigma + \J},\]
one has 
\[[0] = \left[\prod_{\substack{\rho \text{ such that}\\ \langle m,u_\rho\rangle<0}}(1-x_\rho)^{-\langle m,u_\rho\rangle}-\prod_{\substack{\rho \text{ such that}\\ \langle m,u_\rho\rangle>0}}(1-x_\rho)^{\langle m,u_\rho\rangle}\right] = [1-(1 -x_{\rho_0})] = [x_{\rho_0}],\]
where the second equality comes from the fact that $[x_\rho] = [0]$ for all $\rho \notin \sigma(1)$ (by definition of $\I_\sigma$) and therefore both products can be taken over only $\rho \in \sigma(1)$.  Thus, in this case as well, we have $x_{\rho_0} \in \I_{\sigma} + \J$.

In light of the previous paragraph, for any maximal cone $\sigma$ we can write $x_{\rho_0} = y_\sigma + z_\sigma$
for some $y_\sigma \in I_\sigma$ and $z_\sigma \in \J$.  Multiplying this equation over all maximal cones $\sigma$ shows that
\[x_{\rho_0}^k \in \I + \J,\]
where $k$ is the number of maximal cones of $\Sigma$.  In other words, $[x_{\rho_0}] \in \Z[x_\rho]/(\I + \J)$ is nilpotent, from which it follows that $[1 - x_{\rho_0}]$ is a unit, as claimed.

We have therefore proven that the localization map \eqref{eq:localization} is an isomorphism.  By exactness of localization, we have canonical isomorphisms
\[\left(\frac{\Z[x_\rho \; | \; \rho \in \Sigma(1)]}{\I + \J} \right)_t \cong \frac{\left(\Z[x_\rho \; | \; \rho \in \Sigma(1)\right)_t}{(\I + \J)_t} \cong \frac{\left(\Z[x_\rho \; | \; \rho \in \Sigma(1)\right)_t}{\I_t + \J_t},\]
and Corollary~\ref{cor:surj1} shows that $\varphi_t$ induces an isomorphism
\[\frac{\left(\Z[x_\rho \; | \; \rho \in \Sigma(1)\right)_t}{\I_t + \J_t} \cong \uPE(\Sigma).\]
When combined with \eqref{eq:factorphi}, we have now factored $\varphi$ as a composition
\[\Z[x_\rho] \rightarrow \frac{\Z[x_\rho]}{\I + \J} \xrightarrow{\cong} \left(\frac{\Z[x_\rho]}{\I + \J} \right)_t \xrightarrow{\cong} \uPE(\Sigma).\]
Since all but the first map in this composition are isomorphisms, we conclude that the kernel of $\varphi$ is equal to the kernel of the first map, which is $\I + \J$, as claimed.
\end{proof}

Pulling everything together, we now prove the two main results of this paper.

\begin{proof}[Proof of Theorem~\ref{thm1}]
Proposition~\ref{cor:invert} tells us that $\phi$ is surjective, while Proposition~\ref{prop:unlocalize} tells us that $\ker(\phi)=\I+\J$. Together, these prove Theorem~\ref{thm1}.
\end{proof}

\begin{proof}[Proof of Theorem~\ref{thm2}]
Work of Brion and Vergne \cite{BrionVergne} implies that $\PE(\Sigma)\cong K_T(X_\Sigma)$, where the latter is the $T$-equivariant $\K$-ring of $X_\Sigma$; see also a more general result of Anderson and Payne \cite[Theorem~1.6]{AndersonPayne}, which spells out the isomorphism more explicitly. We note that $\PE(\Sigma)\cong K_T(X_\Sigma)$ is a module over the group algebra $\Z[M]$---which can be identified with the representation ring of $T$---where the action on $\PE(\Sigma)$ is via the ring homomorphism \[\ZZ[M] \to \PE(\Sigma), \qquad m\mapsto e^m.\] That $\uPE(\Sigma)\cong K(X_\Sigma)$ then follows from a result of Merkurjev \cite[Corollary~4.4]{Merkurjev}, which shows that $K(X_\Sigma)$ is the quotient of $K_T(X_\Sigma)$ by the ideal generated by the image under  $\ZZ[M] \to \PE(\Sigma)$ of the kernel of the ``rank'' map $\Z[M]\rightarrow\Z$ that sends a linear combination to the sum of its coefficients.  In our notation, that ideal of $\PE(\Sigma)$ is simply $\langle e^\ell-1\mid\ell\in\L(\Sigma)\rangle$. Equipped with the isomorphism $\uPE(\Sigma)\cong K(X_\Sigma)$, Theorem~\ref{thm2} then follows from Theorem~\ref{thm1}.
\end{proof}

\bibliographystyle{alpha}
\bibliography{references}

\end{document}